\documentclass[11pt]{amsart}
\usepackage[margin=1.1in]{geometry}
\usepackage{amsmath,amssymb,amsthm,mathtools,mathrsfs}
\usepackage{enumitem}
\usepackage[dvipsnames]{xcolor}
\usepackage{hyperref}
\hypersetup{colorlinks=true,linkcolor=blue,citecolor=blue,urlcolor=blue}
\usepackage{microtype}
\newtheorem{theorem}{Theorem}[section]
\newtheorem{proposition}[theorem]{Proposition}
\newtheorem{lemma}[theorem]{Lemma}

\theoremstyle{definition}
\newtheorem{remark}[theorem]{Remark}

\newtheorem{example}[theorem]{Example}

\DeclareMathOperator{\rank}{rank}
\DeclareMathOperator{\Span}{Span}
\DeclareMathOperator{\Stab}{Stab}
\newcommand{\C}{\mathbb C}
\newcommand{\R}{\mathbb R}
\newcommand{\Z}{\mathbb Z}
\newcommand{\Kpi}{\ensuremath{K(\pi,1)}}

\title[Non-asphericity of strata of differentials]{Non-asphericity of strata of genus-one differentials and stability spaces}

\author{Dawei Chen}
\address{Department of Mathematics, Boston College, Chestnut Hill, MA 02467, USA}
\email{dawei.chen@bc.edu}

\author{Jingyin Huang}
\address{Department of Mathematics, The Ohio State University,
231 W. 18th Ave, Columbus, OH 43210}
\email{huang.929@osu.edu}

\author{Yu Qiu}
\address{Yau Mathematical Sciences Center and Department of Mathematical Sciences, Tsinghua University, 100084 Beijing, China \& Beijing Institute of Mathematical Sciences
and Applications, Yanqi Lake, Beijing, China}
\email{yu.qiu@bath.edu}

\author{Fei Yu}
\address{School of Mathematical Sciences, Zhejiang University, Hangzhou, 310058, China}
\email{yufei@zju.edu.cn}

\date{\today}


\begin{document}

\begin{abstract}
We show that when the number of zeros or poles is at least four, every connected component of the strata of differentials in genus one with prescribed zero and pole orders is not an orbifold $K(\pi,1)$. For quadratic differentials, this provides infinitely many counterexamples to a conjecture attributed to Kontsevich, as well as to a folklore conjecture concerning the contractibility of spaces of Bridgeland stability conditions.
\end{abstract}
\maketitle

\section{Introduction}
\label{sec:intro}

\subsection{Strata of differentials and the \texorpdfstring{$K(\pi,1)$}{K(pi,1)}-problem}

Holomorphic differentials on Riemann surfaces correspond to translation surfaces with conical singularities, where the zero orders of the differentials determine the associated cone angles. Similarly, meromorphic differentials and $k$-differentials correspond to translation surfaces of infinite area and $\frac{1}{k}$-translation structures, respectively. The loci of differentials with prescribed orders of zeros and poles stratify the moduli spaces of differentials. Up to scalar multiplication, $k$-differentials correspond to pluricanonical divisors, which govern the intrinsic geometry of the underlying complex algebraic curves. These interrelated yet differently motivated viewpoints make the study of differentials a central topic at the intersection of many research fields.

Over the past few decades, remarkable progress on differentials and their moduli spaces has been achieved in diverse directions, including dynamical invariants of translation surfaces, the classification of linear subvarieties, compactifications of strata of differentials, intersection-number and cycle-class computations, and applications to birational geometry. We refer the reader to \cite{Zo06, Mo13, Wr15, Ch17, Ma18, BCGGMk, Fi24, AM24, D25, BC25, CM26, CY26} for introductions and recent surveys.

The strata of differentials with prescribed orders of zeros and poles are not always connected. Following a series of works \cite{KZ03, La04H, La04S, La08, CM14, Bo15, CG22, AA26}, connected components of strata of differentials are now classified in all cases. Nevertheless, the topology of these connected components remains poorly understood. In \cite[Section 4]{KZ97}, Kontsevich and Zorich conjectured that every connected component of every stratum of holomorphic differentials has the homotopy type of a $K(\pi,1)$, with orbifold fundamental group commensurable with a mapping class group. A version including strata of quadratic differentials was attributed to Kontsevich in \cite[Section 1.1.2]{La08}. It is also natural to ask the same question for strata of meromorphic and $k$-differentials.

There are several motivations behind the $K(\pi,1)$-conjecture for strata of differentials. First, the moduli space $\mathcal M_{g,n}$ of genus $g$ curves with
$n$ marked points is a $K(\pi,1)$, as its orbifold universal cover is the Teichm\"uller space $\mathcal T_{g,n}$, which is contractible, and the orbifold fundamental group of $\mathcal M_{g,n}$ is precisely the mapping class group.

Second, strata of holomorphic differentials parameterize smooth deformations in the miniversal deformation spaces of certain Gorenstein curve singularities \cite{B24, CY25}. A parallel $K(\pi,1)$-problem in singularity theory concerns discriminant complements, equivalently smooth deformation spaces, in versal deformation bases of isolated complete intersection singularities. This problem is often attributed to K. Saito; see \cite[p.~26]{Arn93} and \cite[p.~185]{Lo84}. The $ADE$-singularity cases were established in the classical work of Brieskorn and Deligne \cite{Br73, De72}; see also \cite[Chapter~9]{Lo84}. Further positive cases include Wirthm\"uller's work for $S_k$-type singularities \cite{Wi86}, and the first and fourth authors' work concerning the $U_7,U_8,U_9$ simple space-curve singularities \cite{CY25}.
Combining these viewpoints and results, it follows that the hyperelliptic components and a number of low-genus components of strata of differentials are $K(\pi,1)$ \cite{LM14, Gi24, Gi25, CY25}.

When there are sufficiently many simple zeros, the fundamental group and low-degree cohomology groups of the strata can also be described \cite{KQ20,Qi25, Sa25, CL26, T26}.

Other than the aforementioned cases, the $K(\pi,1)$-conjecture and the homotopy groups of strata of differentials remain largely unknown. In this paper, we consider the strata of differentials in genus one. Our main result below shows that when the number of zeros or poles is at least four, every connected component of the strata of differentials in genus one is not a $K(\pi,1)$. This provides infinitely many counterexamples to the $K(\pi,1)$-conjecture for strata of differentials.

We first set up some notation and preliminaries. Let
\[
 \mu=(m_1,\ldots,m_s; -n_1,\ldots,-n_t),\qquad m_i,n_j>0,
\]
be a tuple of nonzero integers such that
\[
 \sum_{i=1}^s m_i=\sum_{j=1}^t n_j.
\]
Since the canonical bundle is trivial on an elliptic curve, the strata of genus-one $k$-differentials of type $\mu$ are isomorphic for all $k$. We thus denote by
\[
 \mathcal P_1(\mu)
\]
the \emph{projectivized stratum of genus-one $k$-differentials of type $\mu$} for all $k$, i.e.\ the space of principal divisors of type
\[
m_1 z_1 + \cdots + m_s z_s - n_1 p_1 -\cdots - n_t p_t
\]
on elliptic curves, where all zeros $z_i$ and poles $p_j$ are pairwise distinct. Similarly, we write $\mathcal H_1(\mu)$ for the corresponding \emph{unprojectivized} stratum of genus-one $k$-differentials of type $\mu$, again independent of $k$. Since the $K(\pi,1)$-property is preserved under finite covers, we label all zeros and poles as marked points. Moreover, note that
\[\mathcal P_1(\mu)\cong \mathcal P_1(-\mu).\]
Hence, we can reverse the zeros and poles if needed.

It is easy to see that $\mathcal P_1(\mu)$ is nonempty for all $\mu \neq (1,-1)$. Hence, from now on we always assume that $\mu \neq (1,-1)$. Write
\[
 d_\mu=\gcd(m_1,\ldots,m_s,n_1,\ldots,n_t).
\]
The connected components of $\mathcal P_1(\mu)$ are classified by the \emph{torsion number} $e$, where $e$ is a positive divisor of $d_\mu$. More precisely, if $e$ is the largest divisor of $d_\mu$ such that
\[
  (m_1/e) z_1 + \cdots + (m_s/e) z_s - (n_1/e) p_1 -\cdots - (n_t/e) p_t \sim 0,
\]
then the corresponding locus in $\mathcal P_1(\mu)$ forms a connected component with torsion number $e$. We refer the reader to \cite[Section 3.2]{CC14}, \cite[Section 4]{Bo15}, \cite[Section 3]{T18},
 and \cite[Section 3.4]{CG22} for more details about the torsion number as well as its flat-geometric interpretation.

For our purpose, it is convenient to consider
the \emph{primitive tuple}
\[
 \bar\mu=(\bar m_1,\ldots,\bar m_s;-\bar n_1,\ldots,-\bar n_t),
\]
where $\bar m_i = m_i/d_{\mu}$ and $\bar n_j = n_j/d_{\mu}$. Then connected components of strata of genus-one differentials can be recorded in the notation
\[
  \mathcal P_1(d\bar \mu)
\]
for $d>0$ and $\bar\mu$ primitive, except that for $d=1$ and $\bar \mu = (1,-1)$ the component is empty, which has already been excluded by our preceding assumption. Here the integer $d$ records the exact order of the primitive Abel divisor. For example, $\mathcal P_1(6(3;-2,-1))$ denotes the component where $3z_1-2p_1-p_2$ has exact order $6$. We say that connected components of type $\mathcal P_1(\bar \mu)$ are \emph{primitive}.

\begin{theorem}
\label{thm:main}
If $s\ge 4$ or $t\ge 4$,
then every primitive connected component $\mathcal P_1(\bar \mu)$ of projectivized strata of genus-one differentials is not a \Kpi.

If $s\ge 3$ or $t\ge 3$,
then every non-primitive connected component $\mathcal P_1(d\bar \mu)$ with $d\ge 2$ is not a \Kpi.
\end{theorem}

The same conclusion applies to the unprojectivized strata $\mathcal H_1(d\bar\mu)$ of genus-one differentials. Indeed, $\mathcal H_1(d\bar\mu)\to \mathcal P_1(d\bar\mu)$ is a $\C^*$-bundle. We use the following form of the homotopy long exact sequence for a locally trivial fibration $F\hookrightarrow E\to B$:
\begin{equation}\label{eq:fibration-les}
 \cdots\to \pi_i(F)\to \pi_i(E)\to \pi_i(B)
 \to \pi_{i-1}(F)\to\cdots.
\end{equation}
If $\mathcal H_1(d\bar\mu)$ were a \Kpi, then \eqref{eq:fibration-les} would imply that $\pi_i(\mathcal P_1(d\bar\mu))=0$ for $i\ge3$, and that $\pi_2(\mathcal P_1(d\bar\mu))$ injects into $\pi_1(\C^*)\cong\Z$. Since the strata have finite-dimensional CW type, a nonzero $\pi_2$ would make the universal cover of $\mathcal P_1(d\bar\mu)$ a finite-dimensional CW model for $K(\Z,2)$, which is homotopy equivalent to $\mathbb C\mathbb P^{\infty}$ and has nontrivial cohomology in arbitrarily high degrees, a contradiction. Thus $\mathcal P_1(d\bar\mu)$ would be a \Kpi, so non-asphericity of $\mathcal P_1(d\bar\mu)$ implies non-asphericity of $\mathcal H_1(d\bar\mu)$. Theorem~\ref{thm:main} therefore implies the following unprojectivized non-asphericity statement.

\begin{theorem}
\label{thm:unproj}
If $s\ge 4$ or $t\ge 4$,
then every primitive connected component $\mathcal H_1(\bar \mu)$ of unprojectivized strata of genus-one differentials is not a \Kpi.

If $s\ge 3$ or $t\ge 3$,
then every non-primitive connected component $\mathcal H_1(d\bar \mu)$ with $d\ge 2$ is not a \Kpi.
\end{theorem}

\subsection{Stability conditions and contractibility}

Our result also has implications for Bridgeland stability conditions. The space
$\Stab(\mathcal D)$ of stability conditions on a triangulated category
$\mathcal D$ is a central object in modern geometry and representation theory.
Bridgeland raised the question of whether these spaces are always contractible
\cite[Remark~3.4]{Bri09Survey}. For K3 surfaces, his earlier conjecture predicts
that the distinguished connected component is simply connected
\cite[Conjecture~1.2]{Bri08}. In Picard rank one, simple connectivity is
equivalent to contractibility in view of Bridgeland's covering description, and
Bayer--Bridgeland proved that the component is indeed contractible
\cite{BB17}; while for higher Picard rank, whether the universal cover of the
base $\mathcal P_0^+(X)$ is contractible remains open
\cite[Remark~1.5(a)]{BB17}. Together with the examples recalled below, these
results support the broader folklore expectation that stability spaces, or at
least their connected components, should be contractible.

Bridgeland showed that, for $2$-Calabi--Yau categories arising from Kleinian
singularities, certain connected components of their stability spaces are covering
spaces of regular orbit spaces of finite or affine Dynkin type \cite{Bri09}. In the
finite Dynkin case, combining this description with Brav--Thomas' result on the
faithfulness of the braid group action \cite{BT11} and Deligne's
$K(\pi,1)$-result \cite{De72} yields contractibility of the corresponding
stability component. In affine type $A$, the same conclusion follows from
Ishii--Ueda--Uehara's faithfulness result \cite{IUU10} and Okonek's
$K(\pi,1)$-result \cite{Oko79}. More generally, together with Bridgeland's
covering description, Paolini--Salvetti's proof of the $K(\pi,1)$-conjecture
for all affine Artin groups \cite{PS21} shows that, for the affine Dynkin
stability components associated with Kleinian singularities, simple connectedness
is equivalent to contractibility.

The third author and Woolf proved that any finite-type component of the space of stability conditions is always contractible \cite{QW18}. This covers the Dynkin case studied in \cite{Bri09} and provides an alternative proof of Deligne's result \cite{De72}; the types $B,C,F,G,H,I$ require additional arguments \cite{QZ23}. More examples of finite-type components can be found in \cite{BPP16, AW22}; see also the survey \cite{Heng24}.

Li established the contractibility of the stability space for $\mathbb P^2$ \cite{Li17}, and an alternative proof via the global dimension function was given in \cite{FLLQ23}. This approach is expected to extend to certain topological Fukaya categories \cite{Q25}. By \cite{HKK17}, the stability spaces arising from these categories correspond to moduli spaces of Teichm\"uller-framed quadratic differentials with only exponential-type singularities, rather than the classical zeros and poles considered in this paper.

Further examples of contractible stability spaces from algebraic geometry include smooth projective curves of positive genus \cite{Macri07}, certain projective varieties \cite{FLZ22}, and a product of three elliptic curves \cite{HaidenSung24}.

There is a correspondence between quadratic differentials and stability conditions, first established by Bridgeland--Smith \cite[Thm.~1.2]{BS15} and Haiden--Katzarkov--Kontsevich \cite[Thm.~5.3]{HKK17} and later generalized in various ways; see \cite[Thm.~1.1]{BMQS24} and \cite[Thm.~1.1]{CHQ23}.
More precisely, after the usual framing or enhancement choices in these correspondences, a connected component of such a stratum is covered by a connected component of the stability space of a suitable triangulated category, such as the corresponding topological Fukaya category in \cite{HKK17}.
In the meromorphic case, provided the polar type is not $(-2)$, there are infinitely many choices of such a triangulated category, by the constructions of \cite{CHQ23}, where the category may be chosen Calabi--Yau, and \cite{BMQS24}.
Thus, after interpreting $\mathcal H_1(\mu)$ as a stratum of genus-one quadratic differentials, Theorem~\ref{thm:unproj} produces infinitely many strata of quadratic differentials for which a corresponding stability space has a connected component that is not a \Kpi; here we use that covering spaces have the same higher homotopy groups.

\begin{theorem}
\label{thm:main2}
Regard $\mathcal H_1(\mu)$ as the corresponding moduli space of quadratic differentials, and let $\mathrm{Quad}_1^{\circ}(\mu)$ be a connected component of $\mathcal H_1(\mu)$ that is not a \Kpi, as in Theorem~\ref{thm:unproj}. Then there exists a triangulated category $\mathcal D$, depending on the component but not necessarily uniquely determined, such that a connected component of $\Stab(\mathcal D)$ is a covering space of $\mathrm{Quad}_1^{\circ}(\mu)$. In particular, this component of $\Stab(\mathcal D)$ is not a \Kpi.
\end{theorem}

\subsection{Proof strategy and arrangement complements}

In order to prove Theorem~\ref{thm:main}, we first reduce the problem to a fixed elliptic curve $E$. Using the group structure of self-products of $E$ and associated universal covers, the problem can be further reduced to studying the complement of a hyperplane arrangement whose hyperplanes arise from the linear equivalence relation among the zeros and poles, together with their collision conditions. Finally, we apply a Zariski-section obstruction to prove the desired non-asphericity. We also explain how this obstruction is related to the Dimca--Papadima theorem on generic sections of aspherical arrangements \cite{DP03} and hypersolvable arrangements, and record in Appendix~\ref{app:half-space} an alternative proof using Yoshinaga's half-space criterion \cite{Y24}.

There is considerable literature on the $K(\pi,1)$-problem for hyperplane arrangements, which is closely related to the other $K(\pi,1)$-problems discussed above. The non-asphericity results mentioned previously fit into a line of work on arrangement complements \cite{Hat75,R97,PS02,DP03,Y08,Y24}. On the other hand, positive results are known for simplicial arrangements, fiber-type arrangements, certain reflection arrangements, and arrangements satisfying suitable curvature tests \cite{De72,Te86,Be15,PaS25,Fa95,GH25}. Since the self-products of elliptic curves considered here are abelian varieties, we also mention general positive \Kpi~results for fiber-type abelian arrangements \cite{BD24}.

When the numbers of zeros and poles are both small, the \Kpi~status of $\mathcal P_1(d\bar \mu)$ is mixed; for example, when $s=t=1$, each stratum is a cover of the modular curve and hence is a \Kpi. Moreover, in genus one, a differential with a nonempty set of zeros must also have poles. Therefore, it would be interesting to look for a stratum of holomorphic differentials in higher genus that is not a \Kpi. We plan to treat these questions in future work.

The paper is organized as follows. In Section~\ref{sec:fixed}, we reduce the problem from strata to a fixed elliptic curve and explain how non-asphericity transfers back to the corresponding components. In Section~\ref{sec:localization}, we recall the localization obstruction for arrangement complements and prove the Zariski-section obstruction. In Section~\ref{sec:primitive}, we apply this obstruction to the identity torsion fiber and prove the primitive case of Theorem~\ref{thm:main} in Proposition~\ref{prop:identity}. In Section~\ref{sec:nonprimitive}, we treat nonzero torsion fibers and prove the non-primitive case of Theorem~\ref{thm:main} in Proposition~\ref{prop:nonzero}. Appendix~\ref{app:half-space} gives an alternative proof of the local arrangement obstructions via Yoshinaga's half-space criterion. Throughout the paper, for moduli spaces regarded as orbifolds, their fundamental groups and universal covers are always considered in the orbifold sense.

\subsection*{Acknowledgements}
The collaboration between D.C. and J.H. began during the 2026 Cornell Topology Festival. We thank the organizers and Xin Zhou for the invitation and hospitality. During the preparation of this work, D.C. was supported by NSF grant DMS-2301030 and the Simons Foundation's Travel Support for Mathematicians program; J.H. was supported by a Sloan Fellowship and NSF grant DMS-2305411; Y.Q. was supported by the National Natural Science Foundation of China under grant number 12425104; F.Y. was supported by the Fundamental Research Funds for the Central Universities 226-2024-00136.

The proof using Yoshinaga's criterion in Appendix~\ref{app:half-space} arose first from an iterative, guided conversation with ChatGPT (version 5.5 Plus). The authors verified, corrected, and revised the proof, and take responsibility for the final content. The proof in Section~\ref{sec:localization} using the Zariski-section obstruction was then obtained by the authors and is the proof used in the main text. The appendix is retained because the Yoshinaga-based argument provides a complementary viewpoint. 

\section{From strata to a fixed elliptic curve}
\label{sec:fixed}

In this section we reduce the proof of the main result to a fixed elliptic curve $E$. Choose a zero index $a\in \{1,\ldots, s\}$ and use the corresponding zero as the origin of $E$. Put
\[
 I_a=\{1,\ldots,s\}\setminus\{a\}.
\]
Define the \emph{primitive zero-normalized Abel map}
\[
 \bar\Phi_a^z\colon E^{I_a}\times E^t\longrightarrow E,
 \qquad
 \bar\Phi_a^z((z_i),p_1,\ldots,p_t)
 =\sum_{i\in I_a}\bar m_i z_i-
   \sum_{j=1}^t \bar n_j p_j.
\]
The full Abel equation for the original signature $\mu$
is
\[
 d_\mu\bar\Phi_a^z=0.
\]
Hence the full Abel kernel decomposes into \emph{torsion fibers}
\[
 X_a^{z,\tau}(E)=(\bar\Phi_a^z)^{-1}(\tau),
\]
where $\tau\in E[d_\mu]$ is a $d_\mu$-torsion point.
Let $\Delta_a^z$ be the union of \emph{collision divisors} defined by
\[
 z_i=0,
 \quad p_j=0,
 \quad z_i=z_{i'},
 \quad p_j=p_{j'},
 \quad z_i=p_j.
\]
Set
\[
 F_a^{z,\tau}(E)=X_a^{z,\tau}(E)\setminus\Delta_a^z.
\]

For a fixed $E$, the fibers $F_a^{z,\tau}(E)$ are indexed by individual torsion points $\tau\in E[d_\mu]$. If $E$ varies in the modular curve, then for every divisor $r$ of $d_\mu$, the union of the fibers over all torsion points $\tau$ of exact order $r$ is monodromy-invariant and gives the corresponding component of the projectivized stratum with torsion number $d_\mu/r$. In particular, $r=1$, equivalently $\tau=0$, gives the component naturally identified with the primitive stratum $\mathcal P_1(\bar\mu)$.

Moreover, for every $\tau\in E[d_\mu]$, the space $F_a^{z,\tau}(E)$ is connected whenever it is nonempty. The only empty case is
$\tau=0$ and $\bar\mu=(1,-1)$; if this occurs for a non-primitive signature, it is just the empty primitive fiber and is not used below. Indeed, the remaining coefficients in $\bar\Phi_a^z$ are primitive, since any common divisor of them also divides $\bar m_a=\sum_j\bar n_j-\sum_{i\ne a}\bar m_i$; hence $X_a^{z,\tau}(E)$ is a connected translate of a subtorus. In every nonempty case, the collision divisor restricts to a proper analytic subset of this torus, and its complement is connected.

\begin{proposition}\label{prop:transfer}
If $F_a^{z,\tau}(E)$ is not a \Kpi, then the corresponding connected component of the projectivized stratum of genus-one differentials is not a \Kpi.
\end{proposition}

\begin{proof}
Let $r=\operatorname{ord}(\tau)$, and let $\mathcal Y_1(r)$ be the modular curve of pairs $(E,\eta)$, where $\eta\in E$ has exact order $r$ (for $r=1$, $\eta=0$). Using the marked zero $z_a$ as the origin, the projectivized stratum component with torsion number $d_\mu/r$ is the relative complement over $\mathcal Y_1(r)$ whose fiber over $(E,\eta)$ is
\[
 (\bar\Phi_a^z)^{-1}(\eta)\setminus\Delta_a^z.
\]
In particular, the fiber over $(E,\tau)$ is the given space $F_a^{z,\tau}(E)$. This relative complement is locally trivial: locally, after trivializing the universal elliptic curve as a group torus and making the torsion section $\eta$ constant, the Abel fiber and the collision divisors form a constant relative subtorus arrangement.

Since $\mathcal Y_1(r)$ is a \Kpi, if the stratum component were a \Kpi, the homotopy long exact sequence \eqref{eq:fibration-les} would force $\pi_i(F_a^{z,\tau}(E))=0$ for all $i\ge2$. Thus the fiber would be a \Kpi, contradicting the hypothesis. Therefore the stratum component is not a \Kpi. If desired, one can make this argument entirely at the level of ordinary topological spaces by adding sufficiently high full level structure; on the resulting finite cover, the same map is a locally trivial fiber bundle over a manifold instead of an orbifold.
\end{proof}

\section{Localization and the Zariski-section obstruction} \label{sec:localization}

Viewing products of elliptic curves as abelian varieties, in this section we first reduce questions about abelian arrangements to questions about hyperplane arrangements. We then record a Zariski-section obstruction for central hyperplane arrangements, which will be the main arrangement-theoretic input in the proof.

Let $X$ be a translate of a complex abelian variety and let $\mathcal B$ be a finite arrangement of translated codimension-one abelian subvarieties. When such an arrangement is obtained by restricting divisors, we tacitly decompose the restrictions into connected codimension-one components and discard empty or non-divisorial pieces.

A \emph{layer} $L$ is a connected component of an intersection of some members of $\mathcal B$. Choose any $p\in L$. The \emph{tangent localization} at $L$ is the central hyperplane arrangement in the normal space $T_pX/T_pL$ consisting of the images of $T_pB$ in the quotient for all $B\in\mathcal B$ with $L\subset B$; members meeting $L$ properly are not included. This definition is independent of $p\in L$ up to the natural identifications by translation.

\begin{proposition}\label{prop:abelian-localization}
For a finite abelian arrangement $\mathcal B$ in $X$, if
\[
 M(\mathcal B)=X\setminus\bigcup_{B\in\mathcal B}B
\]
is a \Kpi, then the complement of every tangent localization of $\mathcal B$ is a \Kpi. Equivalently, if some tangent localization has complement that is not a \Kpi, then $M(\mathcal B)$ is not a \Kpi.
\end{proposition}

\begin{proof}
Lifting the abelian arrangement to the universal cover of $X$, we obtain a locally finite affine hyperplane arrangement $\mathcal A$ in $\C^n$, where $n=\dim X$. Its complement is a covering space of $M(\mathcal B)$. If $M(\mathcal B)$ is aspherical, then this affine arrangement complement $M(\mathcal A)$ is aspherical as well.

Given a layer $L$ of $\mathcal B$, choose a general point $p\in L$, avoiding all members of $\mathcal B$ that do not contain $L$; after passing to the quotient normal space, this makes the actual local germ agree with the tangent localization at $L$, with no extra local hyperplanes. Let $q\in \C^n$ be a lift of $p$. By this choice of $p$, the hyperplanes of $\mathcal A$ through $q$ are exactly the lifts through $q$ of the members of $\mathcal B$ that contain $L$. The \emph{local arrangement $\mathcal A_{q}$ at $q$} is this finite collection of affine hyperplanes. Let $\widetilde L$ be the intersection of all hyperplanes in $\mathcal A_{q}$. The local arrangement descends to a central hyperplane arrangement in $T_{q}\C^n/T_{q} \widetilde L$.

We now apply Oka's localization argument for finite complex hyperplane arrangements, as stated for example in \cite[Lemma~1.1]{P93} or \cite[Proposition 8.1 and Corollary 8.2]{Y26}, only to this finite local arrangement. By local finiteness, there is a sufficiently small ball $B_q$ around $q$ that meets no hyperplanes of $\mathcal A$ except those passing through $q$. Hence
\[
 B_q\cap M(\mathcal A)=B_q\cap M(\mathcal A_q).
\]

Since $\mathcal A_q$ is central at $q$, radial contraction shows that the inclusion $B_q\cap M(\mathcal A_q)\hookrightarrow M(\mathcal A_q)$ is a homotopy equivalence. Composing a homotopy inverse of this inclusion with the inclusion $B_q\cap M(\mathcal A)=B_q\cap M(\mathcal A_q)\hookrightarrow M(\mathcal A)$, we obtain a map $f\colon M(\mathcal A_q)\to M(\mathcal A)$. The composite of $f$ with the natural inclusion $M(\mathcal A)\hookrightarrow M(\mathcal A_q)$ is homotopic to the identity on $M(\mathcal A_q)$. Therefore the latter inclusion induces surjections on all higher homotopy groups. As $M(\mathcal A)$ is aspherical, $M(\mathcal A_q)$ is aspherical.

Finally, $M(\mathcal A_q)$ is the product of the contractible factor $\widetilde L$ with the complement of the central arrangement in $T_{q}\C^n/T_{q}\widetilde L$. Thus the latter complement is aspherical. By the choice of $p$, this central arrangement is identified with the tangent localization in $T_pX/T_pL$.
\end{proof}

Let $\mathcal A$ be a finite arrangement of affine hyperplanes in a complex vector space $V$. Its \emph{intersection poset} is
\[
 L(\mathcal A)=\left\{\bigcap_{H\in\mathcal B}H\ne\emptyset\;\middle|\; \mathcal B\subseteq\mathcal A\right\},
\]
ordered by reverse inclusion; by convention the empty intersection is $V$. Elements of $L(\mathcal A)$ are called \emph{flats} of $\mathcal A$, and a flat $X\in L(\mathcal A)$ has \emph{rank} $\operatorname{codim}_V X$. The rank of $\mathcal A$ is the dimension of the span of the linear parts of defining equations for its hyperplanes. We say that $\mathcal A$ is \emph{central} if $\bigcap_{H\in\mathcal A}H\ne\emptyset$, and \emph{essential} if its rank is $\dim V$. Equivalently, a central arrangement is essential exactly when the intersection of all its hyperplanes is a single point. After translating such a point to the origin, a central arrangement becomes an arrangement of linear hyperplanes; we use this convention when projectivizing central arrangements below.

If $U\subset V$ is a linear subspace, the \emph{restricted arrangement} is
\[
 \mathcal A_U=\{H\cap U:H\in\mathcal A,\ H\cap U\ne U\}.
\]
Following \cite[Section~5]{DP03}, we say that $U$ is \emph{$L_k(\mathcal A)$-generic} if
\[
 \operatorname{codim}_V X=
 \operatorname{codim}_U(X\cap U)
\]
for every flat $X\in L(\mathcal A)$ of rank at most $k+1$. In particular, when $U$ is a hyperplane, this is equivalent to saying that $U$ contains no flat of $\mathcal A$ of rank at most $k+1$. Thus the notation $L_2$-generic below refers to flats of rank at most three.

\begin{lemma}\label{lem:zariski-section}
Let $\mathcal A$ be a central arrangement in a complex vector space $V$ with $\dim V\ge4$, and let $U\subset V$ be an $L_2(\mathcal A)$-generic subspace such that the restricted arrangement $\mathcal A_U$ is essential and $\dim U\ge3$. Then the inclusion
\[
 M(\mathcal A_U)\hookrightarrow M(\mathcal A)
\]
induces an isomorphism on fundamental groups.
\end{lemma}

This is a mild form of the Zariski theorem for arrangement complements. Stronger versions are known, for instance \cite[Proposition 14]{DP03}. We only need the form stated above, which follows from Zariski's hyperplane-section theorem for complements \cite{Zar37,HL73}; we include the standard argument for the reader's convenience.

\begin{proof}
Write
\[
 M_{\mathbb P}(\mathcal A)=
 \mathbb P(V)\setminus\bigcup_{H\in\mathcal A}\mathbb P(H).
\]
We use the following standard form of Zariski's theorem: if $D\subset\mathbb P^m$ is a projective hypersurface, $m\ge3$, and $P\cong\mathbb P^2$ is a sufficiently general plane, then
\[
 \pi_1(P\setminus(D\cap P))\cong \pi_1(\mathbb P^m\setminus D).
\]
See \cite{Zar37,HL73}. 
We use the standard stratified form of this theorem for hypersurface complements. For hyperplane arrangements, it only requires the plane to be transverse to the arrangement stratification through codimension three; equivalently, it is enough that the plane meets every projectivized flat of rank at most three in the expected codimension. Higher-rank flats impose no additional condition, since any flat of rank greater than three is contained in a rank-three flat.

We first prove the projectivized version. If $\dim U=3$, Zariski's theorem applied to
\[
 D=\bigcup_{H\in\mathcal A}\mathbb P(H)\subset\mathbb P(V)
\]
and to the plane $P = \mathbb P(U)$ gives
\[
 \pi_1(M_{\mathbb P}(\mathcal A_U))
 \cong
 \pi_1(M_{\mathbb P}(\mathcal A)).
\]
Here the $L_2(\mathcal A)$-genericity of $U$ is exactly the required transversality condition for $P$. If $\dim U>3$, choose a three-dimensional subspace $W\subset U$ that is $L_2(\mathcal A_U)$-generic. This is possible by avoiding the finitely many proper Grassmannian conditions imposed by the flats of $\mathcal A_U$ of rank at most three. The essentiality of $\mathcal A_U$ also implies that $\mathcal A_W=(\mathcal A_U)_W$ is essential. Let $X\in L(\mathcal A)$ have rank at most three. Since $U$ is $L_2(\mathcal A)$-generic, the intersection $X\cap U$ has the same codimension in $U$ as $X$ has in $V$, and gives a flat of $\mathcal A_U$ of the same rank. Since $W$ is $L_2(\mathcal A_U)$-generic, it meets $X\cap U$, hence $X$, in the expected codimension. Thus $W$ is $L_2(\mathcal A)$-generic. Applying the three-dimensional case to $W\subset U$ and to $W\subset V$ gives
\[
 \pi_1(M_{\mathbb P}(\mathcal A_W))
 \cong
 \pi_1(M_{\mathbb P}(\mathcal A_U)),
 \qquad
 \pi_1(M_{\mathbb P}(\mathcal A_W))
 \cong
 \pi_1(M_{\mathbb P}(\mathcal A)),
\]
and hence
\[
 \pi_1(M_{\mathbb P}(\mathcal A_U))
 \cong
 \pi_1(M_{\mathbb P}(\mathcal A)).
\]

It remains to pass from projectivized to central affine complements. Choose a hyperplane $H_0\in\mathcal A$ whose restriction to $U$ is nonzero; this is possible because $\mathcal A_U$ is essential. If $\alpha_0$ is a linear defining form for $H_0$, then $\alpha_0$ trivializes both central complements over their projectivized complements:
\[
 M(\mathcal A)\cong M_{\mathbb P}(\mathcal A)\times\C^*,
 \qquad
 M(\mathcal A_U)\cong M_{\mathbb P}(\mathcal A_U)\times\C^*.
\]
Under these compatible splittings, the projectivized isomorphism on $\pi_1$, together with the same extra $\Z=\pi_1(\C^*)$-factor on both sides, proves the lemma.
\end{proof}

\begin{proposition}\label{prop:zariski-section-obstruction}
Let $\mathcal A$ be a central arrangement in a complex vector space $V\cong\C^r$, with $r\ge4$. Let $U\subset V$ be an $L_2(\mathcal A)$-generic hyperplane such that $\mathcal A_U$ is essential. 
\begin{enumerate}
\item If $\pi_1(M(\mathcal A))$ contains a subgroup isomorphic to $\Z^r$, then $M(\mathcal A_U)$ is not a \Kpi.
\item If the central arrangement $\mathcal A$ is essential, then $M(\mathcal A_U)$ is not a \Kpi.
\end{enumerate}
\end{proposition}

\begin{proof} Consider the first assertion. By the hypothesis and Lemma~\ref{lem:zariski-section}, the group $\pi_1(M(\mathcal A_U))$ contains a subgroup isomorphic to $\Z^r$. 
On the other hand, $M(\mathcal A_U)$ is an affine variety of complex dimension $r-1$. Therefore, $M(\mathcal A_U)$ has the homotopy type of a CW complex of dimension at most $r-1$, see e.g. \cite{AF59} and \cite[\S 7]{Mil63}. If it were a \Kpi, then $\pi_1(M(\mathcal A_U))$ would have cohomological dimension at most $r-1$. This is impossible, since cohomological dimension is monotone under passage to subgroups, while $\operatorname{cd}(\Z^r)=r$. Therefore $M(\mathcal A_U)$ is not a \Kpi.

The second assertion follows from the first one and the following standard auxiliary claim: if $\mathcal A$ is essential and central of rank $r$, then $\pi_1(M(\mathcal A))$ contains a $\mathbb Z^r$-subgroup. We prove the claim by induction on $r$. The case $r=1$ is immediate. Choose independent hyperplanes
\[
 H_\infty,H_1,\ldots,H_{r-1}\in \mathcal A .
\]
Decone $\mathcal A$ with respect to $H_\infty$, and denote the resulting affine arrangement in $\C^{r-1}$ by $d\mathcal A$. By \cite[Proposition 5.1]{OT92},
\[
 M(\mathcal A)\cong M(d\mathcal A)\times \C^*.
\]
Hence
\[
 \pi_1(M(\mathcal A))
 \cong \pi_1(M(d\mathcal A))\times \Z.
\]
The images of $H_1,\ldots,H_{r-1}$ in the decone meet in a point $p$. The local arrangement $(d\mathcal A)_p$ is an essential and central arrangement of rank $r-1$. The localization argument used in the proof of Proposition~\ref{prop:abelian-localization} gives
\[
 \pi_1(M((d\mathcal A)_p))\hookrightarrow \pi_1(M(d\mathcal A)).
\]
By induction applied to $(d\mathcal A)_p$, the group $\pi_1(M(d\mathcal A))$ contains $\Z^{r-1}$, hence $\pi_1(M(\mathcal A))$
contains $\Z^{r-1}\times \Z\cong \Z^r$.
\end{proof}

This is the form of the obstruction used in the sequel. The following remarks only place it alongside nearby arrangement-theoretic criteria.

\begin{remark}
There are two closely related non-$K(\pi,1)$ criteria. First,
Dimca--Papadima \cite{DP03} proved a non-$K(\pi,1)$ result for sufficiently generic sections of essential $K(\pi,1)$-arrangements. In particular, after the usual essentialization, the discussion preceding \cite[Theorem~16]{DP03}, together with \cite[Theorem~16(i) and Remark~17]{DP03}, implies that if $M(\mathcal A)$ is a \Kpi, $U\subsetneq V$ is an $L_2(\mathcal A)$-generic hyperplane, and $\mathcal A_U$ is essential, then $M(\mathcal A_U)$ is not a \Kpi. 

Second, Jambu--Papadima and Papadima--Suciu developed hypersolvable arrangements and their $\Kpi$ criterion \cite{JP98,PS02}. It follows from their work that if $\mathcal A$ is an essential supersolvable central arrangement in $V$ and $U\subset V$ is an $L_2(\mathcal A)$-generic hyperplane with $\mathcal A_U$ essential, then $M(\mathcal A_U)$ is not a \Kpi. Indeed,  hypersolvability and the hypersolvable length $\ell_{hs}$ are determined by the collection $L_{\le 2}$ of flats of rank $\le 2$ \cite[Section 4.6]{PS02}; a supersolvable arrangement $\mathcal D$ is hypersolvable, with $\ell_{hs}(\mathcal D)=\rank(\mathcal D)$ \cite{JP98};
and, for a hypersolvable arrangement $\mathcal D$, the complement $M(\mathcal D)$ is a \Kpi\ if and only if $\mathcal D$ is
supersolvable \cite[Theorem 1.4]{PS02}. Thus, within the hypersolvable class, $\mathcal D$ is a \Kpi-arrangement if
and only if $\ell_{hs}(\mathcal D)=\rank(\mathcal D)$. The $L_2$-generic restriction identifies the rank-two truncation with that of the ambient arrangement, so the hypersolvable length of the restricted arrangement remains the ambient rank, while its rank drops by one.

In the applications of Proposition~\ref{prop:zariski-section-obstruction} in later sections, the ambient arrangement $\mathcal A$ is either a braid arrangement, or has complement equal to a braid-arrangement complement times $\mathbb C^*$. Hence, the ambient arrangement is both $\Kpi$ and supersolvable. Consequently one can use these two criteria in place of Proposition~\ref{prop:zariski-section-obstruction}.
\end{remark}

\begin{remark}
Finally, Yoshinaga's half-space criterion
\cite[Theorem~5.1(2)]{Y24} gives another useful obstruction to
asphericity for complexified finite central real arrangements of rank at least three. It asserts that the existence
of a sign vector which is locally realizable on every flat of rank at most two
but is not globally realizable forces the complexified complement to be
non-\Kpi. Thus, for a \Kpi~arrangement, every sign vector that is locally
realizable on all flats of rank at most two must also be globally realizable.
Yoshinaga notes that supersolvable, and also simplicial, arrangements
satisfy this condition \cite[Remark~5.2]{Y24}. See also \cite{DDP26} for the relation of this
condition to cleanliness.

For the local arrangements arising in Theorem~\ref{thm:main}, we explicitly
construct a sign vector which is locally realizable on every flat of rank at
most two but is not globally realizable. This gives an alternative proof of
the same local non-asphericity obstruction. We present the verification in
Appendix~\ref{app:half-space}; the main proof below instead uses
Proposition~\ref{prop:zariski-section-obstruction}.
\end{remark}

\section{The primitive components}
\label{sec:primitive}

Following the notation in Section~\ref{sec:fixed}, we first treat $\tau=0$, i.e.\ the case of the primitive components in Theorem~\ref{thm:main}. Put $N=s+t-1$, the number of moving points after deleting the base point, and let
\[
 \mathcal A_N=
 \{x_i=0\}_{i=1}^N\cup\{x_i=x_j\}_{i<j}
 \subset \C^N
\]
be the \emph{marked braid arrangement}. Its complement is the configuration space of $N$ ordered points in $\C^*$. For the corresponding moving coefficient vector $w=(w_1,\ldots,w_N)$, put
\[
 h_w=\sum_{i=1}^N w_i x_i.
\]

We next translate the $L_2$-genericity condition for the marked braid arrangement. A flat of $\mathcal A_N$ is obtained by imposing some equations of the form $x_i=0$ and $x_i=x_j$. So a flat is encoded by a \emph{collision partition} of $\{1,\ldots,N\}$: there may also be a \emph{rooted block} $B_0$, whose coordinates are set equal to $0$ in the flat, and the remaining indices are divided into \emph{moving blocks} $B_1,\ldots,B_q$, where coordinates are equal within each block. The dimension of the flat is $q$, hence its rank is $N-q$. On such a flat, the restriction of $h_w$ is
\[
 \sum_{\alpha=1}^q W_\alpha y_\alpha,
 \qquad
 W_\alpha=\sum_{i\in B_\alpha} w_i.
\]
Consequently, for a positive-dimensional flat, containment in $\{h_w=0\}$ is equivalent to every moving block having total weight zero.

\begin{lemma}\label{lem:primitive-generic}
Assume that one side, positive or negative, of the moving coefficient vector $w$ has cardinality at least $4$. Then the hyperplane $\{h_w=0\}\subset\C^N$ is $L_2(\mathcal A_N)$-generic.
\end{lemma}

\begin{proof}
It suffices to show that $\{h_w=0\}$ contains no flat of $\mathcal A_N$ of rank at most $3$. Suppose to the contrary that such a flat exists. Let $q$ be the number of moving blocks in its collision partition. Since one side has cardinality at least $4$, we have $N\ge4$; hence $q\ge N-3\ge1$. Since the flat is contained in $\{h_w=0\}$, every moving block has total weight zero. As all coefficients $w_i$ are nonzero, every moving block contains coefficients of both signs. If one side has cardinality at least $4$, then the opposite side has cardinality at most $N-4$, so there can be at most $N-4$ moving blocks, contradicting $q\ge N-3$.
\end{proof}

\begin{proposition}\label{prop:identity}
If $t\ge4$ and a zero index $a$ is fixed, then the Abel-fiber component $F_a^{z,0}(E)$ is not a \Kpi. Consequently, if $t\ge4$ or $s\ge4$, then the primitive component $\mathcal P_1(\bar\mu)$ of genus-one differentials is not a \Kpi.
\end{proposition}

\begin{proof}
Assume first that $t\ge4$, and choose a zero base point of index $a$. The moving coefficient vector is
\[
 w=\bigl((\bar m_i)_{i\ne a},-\bar n_1,\ldots,-\bar n_t\bigr),
 \qquad N=s+t-1,
\]
so its negative side has cardinality $t\ge4$. Let $U=\{h_w=0\}\subset\C^N$. By Lemma~\ref{lem:primitive-generic}, $U$ is $L_2(\mathcal A_N)$-generic. The coordinate hyperplanes restrict to hyperplanes in $U$ with common intersection $\{0\}$, so $(\mathcal A_N)_U$ is essential. Since $\mathcal A_N$ is central and essential of rank $N\ge4$, Proposition~\ref{prop:zariski-section-obstruction} (2) implies that the complement $M((\mathcal A_N)_U)$ is not a \Kpi.

The deepest collision layer in $X_a^{z,0}(E)$ where all moving points collide with the base point is
\[
 L_a^z=\{z_i=0\ (i\in I_a),\ p_1=\cdots=p_t=0\}.
\]
The tangent Abel equation at this layer is $h_w=0$, and the tangent collision hyperplanes are the restrictions of the marked-braid hyperplanes $x_i=0$ and $x_i=x_j$. Thus $(\mathcal A_N)_U$ is precisely the tangent localization of the primitive Abel fiber at $L_a^z$. Proposition~\ref{prop:abelian-localization} implies that $F_a^{z,0}(E)$ is not a \Kpi. Proposition~\ref{prop:transfer} then transfers this fixed-fiber obstruction to the primitive component $\mathcal P_1(\bar\mu)$. Finally, if $s\ge4$, then $\mathcal P_1(\bar\mu)\cong\mathcal P_1(-\bar\mu)$ in genus one, so replacing $\bar\mu$ by $-\bar\mu$ reduces this case to the one just proved.
\end{proof}

\section{The non-primitive components}
\label{sec:nonprimitive}
We now treat the nonzero torsion fibers, which correspond to the non-primitive components. Fix a marked zero or pole as the \emph{base point}, i.e.\ the origin of an elliptic curve $E$. Let
\[
 w=(w_1,\ldots,w_N),
 \qquad N=s+t-1,
\]
be the primitive tuple obtained from $\bar\mu$ after deleting this base point. Thus all $w_i\ne0$, and
\[
 \sum_{i=1}^N w_i=-w_{\mathrm{base}}\ne0.
\]
Define
\[
 \Phi_w\colon E^N\longrightarrow E,
 \qquad
 \Phi_w(x_1,\ldots,x_N)=\sum_{i=1}^N w_i x_i.
\]
Let $\Delta$ be the \emph{marked collision divisor} defined by $x_i=0$ and $x_i=x_j$. Put
\[
 E^*=E\setminus\{0\},
 \qquad
 X_w^*=\Phi_w^{-1}(E^*)\setminus\Delta
       =E^N\setminus\bigl(\Delta\cup\Phi_w^{-1}(0)\bigr).
\]
Thus $X_w^*$ is itself an abelian-arrangement complement. For $\tau\in E^*$, the fiber over $\tau$ is
\[
 F_\tau(E)=\Phi_w^{-1}(\tau)\setminus\Delta.
\]
We use the same notation $\Phi_w$ for the restrictions of this map to the subsets considered below. The following fibration statement does not use the primitivity of $w$, only the preceding integral-coefficient setup.

\begin{lemma}\label{lem:punctured-fibration}
The map
\[
 \Phi_w\colon X_w^*\longrightarrow E^*
\]
is a locally trivial topological fibration.
\end{lemma}

\begin{proof}
We first describe the \emph{collision stratification} on $E^N$ induced by the divisor $\Delta$ as follows. A \emph{collision partition} records which of the moving points $x_1,\ldots,x_N$ are equal to one another, and which of them are equal to the base point $0\in E$. Thus a stratum is specified by a partition of the index set $\{1,\ldots,N\}$, together with the possibility that one block is declared to be the rooted block. The rooted block, denoted $B_0$, consists of those indices $i$ for which $x_i=0$. The remaining blocks are called moving blocks. If $B_\alpha$ is a moving block, all coordinates $x_i$ with $i\in B_\alpha$ are equal to a common point $y_\alpha\in E$, and the variables $y_\alpha$ are distinct from each other and from $0$.

On such a collision stratum, the Abel map has the form
\[
 \Phi_w=\sum_\alpha W_\alpha y_\alpha,
 \qquad
 W_\alpha=\sum_{i\in B_\alpha} w_i,
\]
where the sum is over the moving blocks. The rooted block contributes nothing, because its common coordinate is $0$.

We claim that, on every collision stratum which meets $\Phi_w^{-1}(E^*)$, the restriction of $\Phi_w$ is a submersion. Indeed, if at least one moving-block coefficient $W_\alpha$ is nonzero, then the map
\[
 (y_\alpha)_\alpha\longmapsto \sum_\alpha W_\alpha y_\alpha
\]
is a submersion: multiplication by a nonzero integer on an elliptic curve is an \'etale covering, hence has surjective differential. On the other hand, if $W_\alpha=0$ for every moving block, then $\Phi_w$ is identically zero on the stratum. Such a stratum is contained in $\Phi_w^{-1}(0)$, and therefore does not meet $\Phi_w^{-1}(E^*)$.

Now fix a point $\tau\in E^*$, and choose a sufficiently small open disk $U\subset E^*$ containing $\tau$. Since $E^N$ is compact, $\Phi_w\colon E^N\to E$ is proper, and hence its restriction
\[
 \Phi_w\colon\Phi_w^{-1}(U)\longrightarrow U
\]
is proper. Moreover, because $U\cap\{0\}=\emptyset$, the hypersurface $\Phi_w^{-1}(0)$ does not meet $\Phi_w^{-1}(U)$.

Consider the stratification of the pair
\[
 \bigl(\Phi_w^{-1}(U),\,
 \Delta\cap \Phi_w^{-1}(U)\bigr)
\]
induced by the collision stratification, which is a Whitney stratification. By the preceding paragraphs, every stratum is mapped submersively to $U$. Together with properness, Thom's first isotopy lemma gives local topological triviality of the stratified pair
\[
 \bigl(\Phi_w^{-1}(U),\,\Delta\cap \Phi_w^{-1}(U)\bigr)
 \longrightarrow U.
\]
Taking complements of the stratified divisor $\Delta$ in this locally trivial family of pairs, which preserves $\Delta$, we obtain a local topological trivialization of
\[
 \Phi_w^{-1}(U)\setminus \Delta
 \longrightarrow U.
\]
Since
\[
 \Phi_w^{-1}(U)\setminus \Delta
 =
 X_w^*\cap \Phi_w^{-1}(U),
\]
this proves that $\Phi_w\colon X_w^*\to E^*$ is locally topologically trivial near $\tau$. As $\tau\in E^*$ was arbitrary, the lemma follows.
\end{proof}

\begin{example}
Take $N=2$ and $w=(1,1)$. Then
\[
 \Phi_w(x_1,x_2)=x_1+x_2,
 \qquad
 X_w^*=\{x_1,x_2\ne0,\ x_1\ne x_2,\;
 x_1+x_2\ne0\}.
\]
For $\tau\in E^*$, the fiber is described by writing
$x_2=\tau-x_1$. Thus
\[
 F_\tau(E)\cong E\setminus\bigl(\{0,\tau\}\cup [2]^{-1}(\tau)\bigr),
\]
where $[2]\colon E\to E$ is multiplication by $2$. The deleted points
$0$ and $\tau$ correspond to $x_1=0$ and $x_2=0$, while the four
points of $[2]^{-1}(\tau)
=\{\tau/2+\eta:\eta\in E[2]\}$ correspond to the collision condition
$x_1=x_2$. As $\tau$ varies in a sufficiently small disk in $E^*$,
these six deleted points move continuously and remain distinct. Isotopy
extension therefore identifies the punctured elliptic curves over nearby
values of $\tau$, giving a direct model for the local triviality of
$\Phi_w\colon X_w^*\to E^*$ in this simple case.
\end{example}

At the origin, the tangent localization of the deleted abelian arrangement $\Delta\cup\Phi_w^{-1}(0)$ is the complexification of
\[
 B(w)=
 \{x_i=0\}_{i=1}^N
 \cup
 \{x_i=x_j\}_{i<j}
 \cup
 \{h_w=0\}
 \subset\R^N.
\]
The origin is not a point of the complement, but it is a layer of the deleted abelian arrangement; hence Proposition~\ref{prop:abelian-localization} applies to this tangent localization. We denote by $M(B(w)_\C)$ the complement of all hyperplanes in this complex arrangement.

Consider the product arrangement
\[
 \widehat{\mathcal A}_{N}=
 \{x_i=0\}_{i=1}^N\cup\{x_i=x_j\}_{i<j}\cup\{u=0\}
 \subset \C^N\times\C_u.
\]
Its complement is $M(\mathcal A_N)\times\C^*$. Let
\[
 U_w=\{u=h_w(x_1,\ldots,x_N)\}
 \subset \C^N\times\C_u.
\]
After identifying $U_w\cong\C^N$ by the coordinates $x_i$, the restricted arrangement $(\widehat{\mathcal A}_N)_{U_w}$ is exactly $B(w)_\C$.

\begin{lemma}\label{lem:nonprimitive-generic}
Assume that one side of the coefficient vector $w$ has cardinality at least $3$. Then $U_w$ is $L_2(\widehat{\mathcal A}_N)$-generic.
\end{lemma}

\begin{proof}
Since $U_w$ is a hyperplane, we need to show that it contains no flat of $\widehat{\mathcal A}_N$ of rank at most $3$. A flat of $\widehat{\mathcal A}_N$ is the product of a flat $X$ of $\mathcal A_N$ with either $\C_u$ or $\{0\}$. No flat of the form $X\times\C_u$ is contained in the graph $U_w$. Thus a flat contained in $U_w$ must have the form $X\times\{0\}$, and this containment is equivalent to $h_w|_X=0$.

If $X\times\{0\}$ has rank at most $3$, then $X$ has rank at most $2$. Let $q$ be the number of moving blocks in the collision partition corresponding to $X$. Since $\operatorname{rank} X=N-q$, we have $q\ge N-2$. Since $h_w|_X=0$, every moving block has total weight zero, and hence contains coefficients of both signs. If one side of $w$ has cardinality at least $3$, then the opposite side has cardinality at most $N-3$. Each zero-sum moving block contains at least one coefficient from this opposite side, so there can be at most $N-3$ such blocks, contradicting $q\ge N-2$.
\end{proof}

\begin{proposition}\label{prop:Bw}
Assume that one side of the coefficient vector $w$ has cardinality at least $3$. Then $M(B(w)_\C)$ is not a \Kpi.
\end{proposition}

\begin{proof}
By Lemma~\ref{lem:nonprimitive-generic}, the graph hyperplane $U_w$ is $L_2(\widehat{\mathcal A}_N)$-generic. Since one side of $w$ has cardinality at least three, we have $N\ge3$, so the ambient vector space $\C^N\times\C_u$ has dimension $N+1\ge4$. The arrangement $\widehat{\mathcal A}_N$ is central and essential, because the hyperplanes $x_i=0$ and $u=0$ have common intersection equal to the origin. The restricted arrangement is essential as well: after identifying $U_w\cong\C^N$, the restrictions of the hyperplanes $x_i=0$ have common intersection equal to the origin. Proposition~\ref{prop:zariski-section-obstruction} (2) applied to $\widehat{\mathcal A}_N$ and $U_w$ gives that the complement of the restricted arrangement is not a \Kpi. Identifying this restricted arrangement with $B(w)_\C$, we obtain the claim.
\end{proof}

Now we can complete the proof of Theorem~\ref{thm:main} for the case of non-primitive components.

\begin{proposition}\label{prop:nonzero}
Assume that, after choosing a base point, one side of the moving coefficient vector $w$ has at least three elements. Then every nonzero torsion fiber $F_\tau(E)$, $\tau\ne0$, is not a \Kpi. Consequently, if $t\ge3$ or $s\ge3$, then every non-primitive component $\mathcal P_1(d\bar\mu)$ of genus-one differentials with $d\ge2$ is not a \Kpi.
\end{proposition}

\begin{proof}
By Proposition~\ref{prop:Bw}, the complement $M(B(w)_\C)$ of the origin tangent localization is not a \Kpi. Since $B(w)_\C$ is the tangent localization at the origin of the deleted abelian arrangement $\Delta\cup\Phi_w^{-1}(0)$, whose complement is $X_w^*$, Proposition~\ref{prop:abelian-localization} implies that $X_w^*$ is not a \Kpi.

By Lemma~\ref{lem:punctured-fibration}, $X_w^*\to E^*$ is a locally trivial fibration over the connected base $E^*$. The base $E^*$ is a \Kpi. If one fiber were a \Kpi, then all fibers would be homeomorphic to a \Kpi; by \eqref{eq:fibration-les}, $X_w^*$ would be a \Kpi, contradicting the preceding paragraph. Hence every fiber over $E^*$, and in particular every nonzero torsion fiber, is not a \Kpi.

Now consider a non-primitive component $\mathcal P_1(d\bar\mu)$ with $d\ge2$. By the discussion in Section~\ref{sec:fixed}, it is covered by fixed-elliptic-curve fibers with torsion value $\tau$ of exact order $d$, hence $\tau\ne0$. If $t\ge3$, choose a zero base point; then the negative side of the moving coefficient vector has cardinality $t$, so the first part applies. If $s\ge3$, the same conclusion follows after zero-pole duality, or equivalently after choosing a pole base point. Finally, Proposition~\ref{prop:transfer} transfers the fixed-fiber obstruction to the corresponding connected component of the stratum of genus-one differentials. This proves the desired claim.
\end{proof}

\appendix

\section{A half-space proof of the local obstructions}
\label{app:half-space}

This appendix gives an alternative proof of the two local arrangement obstructions used in Sections~\ref{sec:primitive} and~\ref{sec:nonprimitive}, based on Yoshinaga's half-space criterion \cite{Y24}. It is logically independent of Proposition~\ref{prop:zariski-section-obstruction}; the localization and transfer arguments are unchanged from the main text.

Let $\mathcal A=\{H_i\}_{i\in I}$ be a finite central real arrangement in a real vector space $V$, with fixed defining forms $H_i=\ker(\alpha_i)$. A \emph{sign vector} is a choice $\epsilon=(\epsilon_i)_{i\in I}\in\{\pm1\}^I$, determining the strict half-space system
\[
 \epsilon_i\alpha_i(x)>0,
 \qquad i\in I.
\]
It is \emph{globally realizable} if this system has a solution in $V$. If $X$ is a flat, the localization $\mathcal A_X$ consists of the hyperplanes containing $X$. The sign vector is \emph{locally realizable at $X$} if the subsystem with $H_i\supset X$ is realizable. We say that $\epsilon$ is \emph{locally realizable up to rank $k$} if this holds for every flat $X$ with $\operatorname{codim}_V X\le k$.

Yoshinaga states the following criterion for central essential arrangements \cite[Theorem~5.1(2)]{Y24}. The form below allows a non-essential arrangement as well, by quotienting the ambient space by the common intersection of all hyperplanes.

\begin{theorem}[Yoshinaga]
\label{thm:yoshinaga-half-space}
If a finite central real arrangement $\mathcal A$ of rank at least three has a sign vector that is locally realizable up to rank two but is not globally realizable, then the complement of the complexification $\mathcal A_\C$ is not a \Kpi.
\end{theorem}

We also use the following elementary homogeneous form of Gordan's alternative: for linear forms $\ell_1,\ldots,\ell_m$ on a real vector space, the strict system $\ell_i(x)>0$ has no solution if and only if there are coefficients $c_i\ge0$, not all zero, such that $\sum_i c_i\ell_i=0$; see \cite[Section~7.8]{Sch86}.

Let $V$ be the ambient real coordinate space of a marked-braid arrangement. For a subcollection of hyperplanes $x_i=0$ and $x_i=x_j$, let $G$ be the graph with vertices the coordinate symbols and the root $0$ that occur in the equations, and with an edge for each equation. For every connected component $\Gamma$ of $G$, the span of its edge forms has dimension $|\Gamma|-1$, where the root is counted as a vertex if it lies in $\Gamma$. If $\Gamma$ does not contain the root, every form in this span has total coefficient sum zero on the coordinates in $\Gamma$; we call such a component a \emph{moving component}.

\subsection*{The primitive local arrangement}

Assume $t\ge4$ and choose a zero base point as in Proposition~\ref{prop:identity}. Write the real coordinates as $(y_i)_{i\in I_a}$ for the non-base zeros and $(x_j)_{j=1}^t$ for the poles. The primitive local arrangement is the restriction of the marked-braid arrangement
\[
 y_i=0,
 \quad x_j=0,
 \quad y_i=y_{i'},
 \quad x_j=x_{j'},
 \quad y_i=x_j
\]
to the real Abel hyperplane
\[
 V_a^z=\left\{\sum_{j=1}^t \bar n_j x_j-
             \sum_{i\in I_a}\bar m_i y_i=0\right\}.
\]
Choose an arbitrary order of the $y_i$'s and consider the sign vector on the restricted arrangement determined by
\[
 x_1>x_2>\cdots>x_t>y_{i_1}>\cdots>y_{i_{s-1}}>0.
\]
If $I_a=\emptyset$, the middle $y$-part is omitted, so the chain is simply $x_1>\cdots>x_t>0$.

\begin{lemma}\label{lem:yosh-primitive-local}
This sign vector on $V_a^z$ is locally realizable up to rank two but is not globally realizable. Hence the complement of this restricted marked-braid arrangement, equivalently the primitive tangent localization, is not a \Kpi.
\end{lemma}

\begin{proof}
Global realizability would imply
\[
 \sum_{j=1}^t \bar n_j x_j>
 \left(\sum_{j=1}^t \bar n_j\right)x_t,
 \qquad
 \sum_{i\in I_a}\bar m_i y_i\le
 \left(\sum_{i\in I_a}\bar m_i\right)x_t.
\]
Since $x_t>0$ and $\sum_j\bar n_j=\bar m_a+\sum_{i\in I_a}\bar m_i>\sum_{i\in I_a}\bar m_i$, this contradicts the defining equation of $V_a^z$.

We now show that every nonrealizable subsystem has rank at least $t-1$ after restriction to $V_a^z$. Let $J$ be a nonrealizable subsystem and let $L_J$ be the span, in the ambient dual space, of its signed marked-braid forms. Put
\[
 h_a^z=\sum_{j=1}^t \bar n_j x_j-
       \sum_{i\in I_a}\bar m_i y_i.
\]
By Gordan's alternative applied on $V_a^z$, a nontrivial nonnegative combination of the restrictions of the forms in $J$ is zero. Equivalently, a nonnegative combination of the ambient forms in $J$ equals a scalar multiple of $h_a^z$. This scalar is nonzero because all chosen marked-braid inequalities are simultaneously satisfied in the ambient chamber. Thus $h_a^z\in L_J$. Restricting to $V_a^z=\ker(h_a^z)$ therefore lowers the rank by one, so the restricted rank of the subsystem is $\dim L_J-1$.

Let $G_J$ be the graph of the marked-braid forms in $J$, and let $P=\{x_1,\ldots,x_t\}$. Since $h_a^z$ has nonzero coefficient at every coordinate and lies in $L_J$, all coordinate vertices occur in $G_J$. For a connected component $\Gamma$ of $G_J$, put $p(\Gamma)=|\Gamma\cap P|$. If $\Gamma$ contains the root, its incidence span has dimension at least $p(\Gamma)$. If $\Gamma$ is moving, then all forms in its incidence span have total coefficient sum zero on $\Gamma$. Since $L_J$ decomposes over the connected components of $G_J$ and $h_a^z\in L_J$, the sum of the coefficients of $h_a^z$ over the coordinates in $\Gamma$ is also zero; hence any moving component meeting $P$ also contains some $y_i$, and its incidence span has dimension $|\Gamma|-1\ge p(\Gamma)$. Summing over components gives
\[
 \dim L_J\ge \sum_\Gamma p(\Gamma)=t.
\]
Thus the restricted rank is at least $t-1\ge3$. Therefore every subsystem of rank at most two is realizable. Theorem~\ref{thm:yoshinaga-half-space} applies.
\end{proof}

If instead $s\ge4$, the same argument applies after interchanging zeros and poles, equivalently after replacing $\bar\mu$ by $-\bar\mu$. Thus Lemma~\ref{lem:yosh-primitive-local} supplies the primitive local obstruction in the full range needed for Proposition~\ref{prop:identity}.

\subsection*{The non-primitive local arrangement}

Let $w=(w_1,\ldots,w_N)$ be the moving coefficient vector of Section~\ref{sec:nonprimitive}. Put
\[
 Z=\{i:w_i>0\},\qquad P=\{i:w_i<0\},
\]
and
\[
 w_Z=\sum_{i\in Z}w_i,
 \qquad
 w_P=\sum_{j\in P}(-w_j).
\]
Since $\sum_{i=1}^N w_i\ne0$, we have $w_Z\ne w_P$. Let $D$ be the side $Z$ or $P$ with larger total absolute weight.

\begin{lemma}\label{lem:yosh-nonprimitive-local}
If $|D|\ge3$, then the complement $M(B(w)_\C)$ is not a \Kpi.
\end{lemma}

\begin{proof}
Assume first that $w_P>w_Z$, so $D=P$; the other case is obtained by replacing $w$ by $-w$, which leaves $B(w)$ unchanged. Choose a total order
\[
 x_{p_1}>x_{p_2}>\cdots>x_{p_{|P|}}>
 x_{z_1}>\cdots>x_{z_{|Z|}}>0.
\]
If one side is empty, the corresponding part of the chain is omitted.
Take the marked-braid signs determined by this chamber, and for the Abel hyperplane $h_w=0$, choose the opposite sign $h_w>0$. In the displayed chamber one has $h_w<0$ throughout: the positive contribution is at most $w_Z x_{p_{|P|}}$, while the negative contribution is at least $w_P x_{p_{|P|}}$, and $w_P>w_Z$. Hence the full system is not globally realizable.

We show that every nonrealizable subsystem has rank at least $|D|$. A subsystem not containing $h_w>0$ is part of a marked braid chamber and is realizable. Hence any nonrealizable subsystem contains $h_w>0$. Let $J$ be the marked-braid part of such a subsystem, with signed defining forms $f_e$, and let $L_J=\Span\{f_e:e\in J\}$. By Gordan's alternative and realizability of the marked-braid part alone, there is a relation
\[
 c_0h_w+
 \sum_{e\in J}c_ef_e=0,
 \qquad c_0>0,
 \qquad c_e\ge0,
\]
so $h_w\in L_J$.

Let $G_J$ be the graph of the marked-braid forms in $J$. Since $h_w$ has nonzero coefficient at every coordinate and lies in $L_J$, every coordinate vertex occurs in $G_J$. For a connected component $\Gamma$ of $G_J$, put $d(\Gamma)=|\Gamma\cap D|$. If $\Gamma$ contains the root, its incidence span has dimension at least $d(\Gamma)$. If $\Gamma$ is moving, every form in its incidence span has total coefficient sum zero on $\Gamma$. Since $L_J$ decomposes over the connected components of $G_J$ and $h_w\in L_J$, the sum of the coefficients of $h_w$ over the coordinates in $\Gamma$ is also zero. Thus any moving component meeting $D$ must also meet the opposite side, and its incidence span has dimension $|\Gamma|-1\ge d(\Gamma)$. Therefore
\[
 \dim L_J\ge \sum_\Gamma d(\Gamma)=|D|.
\]
Adding the Abel form does not increase the rank because $h_w\in L_J$. Thus every nonrealizable subsystem has rank at least $|D|\ge3$, and the sign vector is locally realizable up to rank two. Theorem~\ref{thm:yoshinaga-half-space} applies.
\end{proof}

For the non-primitive components appearing in Theorem~\ref{thm:main}, the dominant-side hypothesis used in Lemma~\ref{lem:yosh-nonprimitive-local} is automatic after a suitable choice of base point. If $t\ge3$, choose a zero base point. Then the negative side is dominant, since
\[
 \sum_{j=1}^t\bar n_j
 =\bar m_a+
  \sum_{i\ne a}\bar m_i
 >
  \sum_{i\ne a}\bar m_i,
\]
and it has cardinality $t$. If $s\ge3$, choose a pole base point, or equivalently apply zero-pole duality. Thus Lemma~\ref{lem:yosh-nonprimitive-local} supplies the non-primitive local obstruction in the range needed for Proposition~\ref{prop:nonzero}. Combined with the tangent-localization and transfer arguments of Sections~\ref{sec:primitive} and~\ref{sec:nonprimitive}, Lemmas~\ref{lem:yosh-primitive-local} and~\ref{lem:yosh-nonprimitive-local} give an alternative proof of Theorem~\ref{thm:main}.

\bibliographystyle{alpha}
\bibliography{biblio}

\end{document}